\documentclass{amsart}
\author{Umar Hayat}
\title{A note on the canonical divisor of the generalised affine Stiefel algebraic varieties}
\date{Monday, December 13, 2014}
\address{Umar Hayat, Mathematics Section, International Center for Theoretical Physics, Trieste, Italy}
\address{Faculty of Engineering Sciences, GIK Institute of Engineering Sciences and Technology, Topi, Khyber Pakhtunkhwa, Pakistan}
\email{umarmaths@gmail.com}
\usepackage{amsmath,bm,graphicx}
\usepackage{subfigure}
\usepackage{amssymb}
\usepackage{amsfonts}
\usepackage{amsthm}
\usepackage{comment}
\newcommand{\GL}{GL}

\newcommand{\Sym}{S}

\newcommand{\I}{I}

\newcommand{\Hom}{Hom}

\newcommand{\Stab}{Stab}

\newcommand{\GSV}{GSV}

\newcommand{\IC}{\mathbb{C}}

\newcommand{\IT}{\mathbb{T}}

\newcommand{\cO}{\mathcal{O}}

\newtheorem{pro}{Proposition}
\newtheorem{thm}{Theorem}

\begin{document}
\begin{abstract}
In this paper we study certain homogeneous spaces, which we call generalised affine Stiefel algebraic varieties. The main aim is to characterise the canonical divisor of generalised affine Stiefel algebraic varieties  in terms of group representations. Affine Stiefel algebraic varieties and in particular $S^{n}$ are two special cases of the generalised affine Stiefel algebraic varieties.   

\end{abstract}



\maketitle

\pagestyle{myheadings}

\markboth{UMAR HAYAT}{The Generalised affine Stiefel Varieties}

\section{Introduction}
\label{sec-intro}
Let $G$ be an algebraic group. If $G$ acts on a  space $M$ transitively, we say that $M$ is a  homogeneous space of $G$. Let $x \in M$, the group elements fixing $x$ is a subgroup of $G$. We call this subgroup stabiliser or isotropy subgroup.  Equivalently,  a homogeneous space is of the form $G/H$, where $G$ is an algebraic group and $H$ a closed subgroup of $G$. Homogeneous spaces play a vital role in the representation theory of the algebraic group because representations are often realised as the space of sections of vector bundles over homogeneous spaces. Many authors \cite{UH,HOC,KM,PR,PR1} have studied homogeneous and  quasi-homogeneous spaces in different contexts. Corti and Reid \cite{RC} studied weighted analogs of the homogeneous spaces.

Let $X$ be an $r\times s$ matrix where $r\leq s$. An affine Stiefel algebraic variety is defined by the following matrix equation
$$
XX^{tr}=I_{r\times r}.
$$\\
Let $X$ and $Y$ be $r\times s$ and $s\times r$ matrices. We define a variety $V \in \IC^{2(r\times s)}$ by the following matrix equation
$$
XY=I_{r\times r}.
$$
We call this variety generalised affine  Stiefel algebraic variety and denote it by $\GSV(r,s)$ . \\
In this paper, we show that  $\GSV(r,s)$ is a homogeneous space, an orbit of a vector which is not a weight vector. 

In section $2$, proposition $1$, we prove that the canonical divisor of the  generalised affine Stiefel algebraic variety  $\GSV(r,s)$ is Cartier. We use a representation-theoretic approach to calculate the canonical class of the $\GSV(r,s)$ in section $3$, Theorem $1$. 

\section{The variety in Equations}
 Let $X$ and $Y$ be the $r\times s$ and $s\times r$ matrices as given below:

\[
 X =
 \begin{pmatrix}
  x_{11} & x_{12} & \cdots & x_{1s} \\
  \vdots  & \vdots  & \ddots  & \vdots \\ 
  x_{r1} & x_{r2} & \cdots & x_{rs} \\
   \end{pmatrix}
   \qquad
\text{and }
\qquad
 Y =
 \begin{pmatrix}
  y_{11} & y_{12} & \cdots  & y_{1r} \\
   \vdots  & \vdots  & \ddots& \vdots    \\
  y_{s1} & y_{s2} & \cdots & y_{sr}\\
 \end{pmatrix},
\]
where $r\leq s.$

We define a variety $V \subset \IC^{2(r\times s)}$ by the equation
$$
XY=\I_{r\times r}. 
$$
If both $X$ and $Y$  are of maximal rank then $V$  has codimension $rs+r(s-r)$. We show in section $3$ that $V$ is a homogeneous space, the orbit of the vector  

\[
\left(
 X_{0} =
 \begin{pmatrix}
  I_{r\times r} & 0_{r\times (s-r)}  \\
     \end{pmatrix}
\text{, }
\quad
 Y_{0} =
 \begin{pmatrix}
  I_{r\times r} \\
  0_{(s-r)\times s}  \\
   \end{pmatrix}   
   \right),
\]
under an action of $G=\GL(r)\times \GL(s)$. This is explained in section \ref{s!Qausi}.

When $X$ and $Y$ are of  maximal rank then we can assume the first minor $ X_{s_{1},\cdots, s_{r}}$ of $X$ is nonzero. We can use that to solve the top $r$ rows of $Y$ in terms of remaining entries of $Y$ and $X$.
 
In all other cases where rank $X\leq r-1$ or rank $Y\leq r-1$, the condition $XY=\I_{r\times r}$  is not satisfied so it does not influence the calculaton of the canonical class of $V$.
\subsection{The canonical class of $\GSV(r,s)$}\label{s!KV}
 If we assume that some $r\times r$  minor  of $X$ is nonzero then we can use that to solve for the $r$ rows of $Y$ and the coordinates will be the remaining entries of $Y$ and $X$. For example, suppose that the first minor $ X_{s_{1},\cdots, s_{r}}$ of $X$ is  nonzero. Then we can write entries of first $r$ rows of $Y$, $y_{11}$,$\cdots$  ,$y_{1r}$, $y_{21}$, $y_{22}$,$\cdots$ ,$y_{2r}$, $\cdots$ $y_{r1}$, $y_{r2}$,$\cdots$ ,$y_{rr}$  in terms of remaining entries of $X$ and $Y$. 
 Similarly, if we assume that the minor $ X_{s_{2},\cdots, s_{r+1}}$ of $X$ is nonzero then we can solve for $y_{21}$,$\cdots$  ,$y_{2r}$, $y_{31}$, $n_{32}$,$\cdots$ ,$y_{3r}$, $\cdots$ ,$y_{r+1,1}$, $y_{r+1,2}$,$\cdots$ ,$y_{r+1,r}$. 
 
 Here we explicitly explain  in terms of coordinates. Let $U_{X_{s_{1},\cdots, s_{r}}\neq 0}$  and \\ $U_{X_{s_{2},\cdots, s_{r+1}}\neq 0}$  be the two charts for $V$ with coordinates \linebreak
$\zeta_{1},\dots, \zeta_{rs}, \xi_{rs+1} \cdots , \xi_{rs+r(s-r)}$ and $\mu_{1},\dots, \mu_{rs}, \mu_{rs+1} \cdots , \mu_{rs+r(s-r)}.$ 
The majority of the coordinates between the two charts are common because these charts differ only by one row of the $Y$ and there are exactly $rs+r(s-r)-r$ coordinates in common. The coordinate transformation matrix  $J$ from one chart to the other is given by
\[
J=
 \begin{pmatrix} 
 I_{rs+r(s-r)-r \times rs+r(s-r)-r} & 0_{rs+r(s-r)-r \times r}\\
 C_{r \times rs+r(s-r)-r} & D_{r\times r}\\
  \end{pmatrix},  
  \] 
where the submatrix   $C_{r \times rs+r(s-r)-r}$ is formed of the partial derivatives of the non-overlapping variables with respect to the overlapping variables and $D$ is the $r\times r$ diagonal matrix whose diagonal entries are $\dfrac{ X_{s_{2}, \cdots, s_{r+1}}}{ X_{s_{1},\cdots, s_{r}}}$  and the determinant of the Jacobian matrix $J$ is $\bigg(\dfrac{ X_{s_{2}, \cdots, s_{r+1}}}{ X_{s_{1},\cdots, s_{r}}}\bigg) ^{r}$.

 Since the minor $X_{s_{1},\cdots, s_{r}}$ is nonzero on the chart  $U_{X_{s_{1},\cdots, s_{r}}}$ and similarly  $X_{s_{2},\cdots, s_{r+1}}$ is invertible on the chart $U_{X_{s_{2},\cdots, s_{r+1}}}$. Suppose

\begin{equation*}\label{s12}
 \sigma_{1, \cdots, r}= \dfrac{d\zeta_{1}\wedge \dots \wedge d\zeta_{rs+r(s-r)}}{(X_{s_{1},\cdots, s_{r}})^{r}} \text{ and similarly }  \sigma_{2, \cdots, r+1}= \dfrac{d\mu_{1}\wedge \dots \wedge d\mu_{rt+r(s-r)}}{(X_{s_{2}, \cdots, s_{r+1}})^{r}}.
 \end{equation*}
 To put $(X_{s_{1},\cdots, s_{r}})^{r}$ in the denominator is a convenient trick to cancel out the deteminant of the Jacobian matrix and will appear again later.  
The sheaf of the canonical differentials is 
 $$
\cO(K_{V})=\bigwedge^{rs+r(s-r)}\Omega^{1}_{V} \text{ and }
\cO(K_{V})\mid_{U_{X_{s_{1},\cdots, s_{r}}\neq 0}}=\cO_{U_{X_{s_{1},\cdots, s_{r}}\neq 0}}\cdot \sigma_{1, \cdots, r} 
$$
 The above calculation of the Jacobian determinant shows that $\sigma_{1,\cdots, r}= \sigma_{2,\cdots, r+1}$ and repeating the same calculation defines $\sigma= \sigma_{i_{1}, \cdots, i_{r}}$  independently of $i_{1}, \cdots, i_{r}$. Since $\sigma_{i_{1}, \cdots, i_{r}}$ is a basis of $\bigwedge^{rt+r(s-r)}\Omega^{1}_{V}$ and has no zeros or poles, exactly because of the $X_{i_{1}, \cdots, i_{r}}$ in the denominator, we have   $$
K_{V}=divisor(\sigma)=0.
$$
In the above discussion we have shown that there is an open cover $\{ (U_{X_{i_{1}, \cdots, i_{r}}}) \}$  for the $\GSV(r,s)$, with transition functions $\dfrac{1}{X_{i_{1}, \cdots, i_{r}}} \in k(U_{X_{i_{1}, \cdots, i_{r}} })^{\ast}$. Furthermore $\dfrac{X_{s_{1},\cdots, s_{r}}^{r}}{X_{s_{2},\cdots, s_{r+1}}^{r}} \in \cO^{*}(U_{X_{s_{1},\cdots, s_{r}} }\cap U_{X_{s_{2},\cdots, s_{r+1}} })= \cO^{*}(U_{X_{s_{1},\cdots, s_{r}} },_{X_{s_{2},\cdots, s_{r+1} } })$.\\
We summarise the discussion in the preceding section in the following proposition. 
\begin{pro}
{\it The canonical divisor of the generalised affine Stiefel algebraic variety $\GSV(r,s)$ is Cartier}.

\end{pro}
\section{The generalised affine Stiefel Variety $\GSV(r,s)$ as a homogeneous space}\label{s!Qausi}
The main goal of this section is to study the variety $V$ as an orbit of a special vector, which is not a weight vector. We fix $G=\GL(r)\times \GL(s)$. Note that $G$ is  a reductive algebraic group. Let $W_{r}$ and $W_{s}$ be the natural representations of $\GL(r)$ and $\GL(s)$ of dimension  $r$ and $s$  respectively.

We wish to define an action of $G=\GL(r)\times \GL(s)$ on the representation $W=\Hom(W_{r},W_{s})\oplus\Hom(W_{s}, W_{r})$ that keeps $V$ invariant. In coordinate-free terms, $X\in\Hom(W_{r},W_{s})$ and $Y\in\Hom(W_{s}, W_{r})$  and the action of $(A,B)\in G$ with $A\in\GL(r)$, $B\in\GL(s)$ is defined  as follows,
\begin{align*}
X &\longmapsto AXB^{-1} \\
Y &\longmapsto BYA^{-1}. \\
\end{align*}
If $X$ and $Y$ are  matrices of maximal rank then  using row and column operations we can put  $X$ and $Y$ in the following form
\[
\left(
 X_{0} =
 \begin{pmatrix}
  I_{r\times r} & 0_{r\times s-r}  \\
     \end{pmatrix}
\text{, }
\quad
 Y_{0} =
 \begin{pmatrix}
  I_{r\times r} \\
  0_{s-r\times r}  \\
   \end{pmatrix}
   \quad
   \right).
\]
The subgroup of $G$ that stabilises the vector $v=\left(  X_{0},\text{ } Y_{0} \right) $ is given by
$$
H=\Stab(v)=\left\{\left( A,B\right) \text{ }\bigg \vert 
 B =
 \begin{pmatrix}
 A & 0 \\ 
 0 & *_{s-r}\\
 \end{pmatrix}
 \right\}, 
$$
where ($*$) means there is no restriction on this block. One can observe that $v$ is not the highest weight vector, it is not even a weight vector. The action of $G$ is transitive and 
$$
V=G/H\simeq G \cdot v \hookrightarrow W.
$$
The generalised affine Stiefel algebraic variety  $\GSV(r, s)$ is a homogeneous space with the natural action of $G$.
\subsection{The Weyl group $W(G)$}
To study the algebraic group $G=\GL(r)\times \GL(s)$ and its representations we use the Weyl group  $W(G)\cong \Sym_{r}\times \Sym_{s}$ which acts as a permutation group. We know from section \ref{s!Qausi}  that  $W=\Hom(W_{r},W_{s})\oplus\Hom(W_{s}, W_{r})$ is a representation of $G$.  The Weyl group acts on $W$ as follows. The group $\Sym_{r}$ acts on any $X\in \Hom(W_{r},W_{s})$ from the left and permutes the rows while $\Sym_{s}$ acts on the right and permutes the columns.
Similarly, $\Sym_{s}$ acts on $Y\in \Hom(W_{s}, W_{r})$ from the left and permutes the rows and $\Sym_{r}$ acts on right and permutes the columns.
\subsection{The torus action and Weyl group}
Let $\IT\subset G$ given below
\begin{multline*}
\IT = \Biggl[  T_{A}=
\begin{pmatrix}
  a_{11} & 0 & \cdots & 0 \\
  0 & a_{22} & \cdots & 0 \\
  \vdots  & \vdots  & \ddots & \vdots \\
  0 & 0 & \cdots & a_{rr} \\
  \end{pmatrix}, \text{  }
T_{B}= 
 \begin{pmatrix}
  b_{11} & 0 & \cdots & 0 \\
  0 & b_{22} & \cdots & 0 \\
  \vdots  & \vdots  & \ddots & \vdots \\
  0 & 0 & \cdots & b_{ss} \\
  \end{pmatrix} 
 \Biggr] 
\end{multline*} 
be the maximal torus of $G$. 
  The maximal torus $\IT$ acts on 
 \[
 X_{11} =
 \begin{pmatrix}
  x_{11} & 0_{1\times s-1} \\
  0_{r-1\times 1} &  0_{r-1\times s-1}\\
   \end{pmatrix}
   \quad
\text{and }
\quad
 Y_{11} =
 \begin{pmatrix}
  y_{11} & 0_{1\times r-1} \\
  0_{s-1\times 1} & 0_{s-1\times r-1}\\
 
 \end{pmatrix}
\] 
as explained in section $3$. Under this action $X_{11}$ and $Y_{11}$ are the weight vectors with weights $
\frac{a_{11}}{b_{11}}$ and  $\frac{b_{11}}{a_{11}}$ respectively.
\subsection{A representation theoretic approach to calculate the canonical class}
For the maximal torus $\IT \subset G$, the canonical differential  $\sigma_{1, \cdots, r}= \\ \dfrac{d\zeta_{1}\wedge \dots \wedge d\zeta_{rs+r(s-r)}}{(X_{s_{1},\cdots, s_{r}})^{r}}$  is a weight vector with weight $1$.  Similarly all the $\sigma= \sigma_{i_{1}, \cdots, i_{r}}$ are weight vectors with  unit weight. But the maximal torus  $\IT \subset G$ does not normalise the stabiliser $H$. If we choose the restricted torus $\IT_{H}=\IT \cap N_{H}$, where $N_{H}$ is the normaliser of $H$ then $\IT_{H}$ 
$$
\left\{  T_{A}, \text{  }
T_{B}= 
 \begin{pmatrix}
  T_{A} &  0_{r\times s} \\
  0_{s\times r}  & *_{s-r\times s-r} \\
    \end{pmatrix}  
 \right\} 
 $$ 
 acts on the canonical differential  $\sigma_{1, \cdots, r}= \dfrac{d\zeta_{1}\wedge \dots \wedge d\zeta_{rs+r(s-r)}}{(X_{s_{1},\cdots, s_{r}})^{r}}$ and  $\sigma_{1, \cdots, r}= \dfrac{d\zeta_{1}\wedge \dots \wedge d\zeta_{rs+r(s-r)}}{(X_{s_{1},\cdots, s_{r}})^{r}}$ is a weight vector with weight 
 $1$.
 
 Let $\mathfrak{h}$ and $\mathfrak{g}$ be the Lie algebras of $H$ and $G$ respectively. The tangent space  $T_{G/H}$ to $G/H$ at the identity $H$ can be identified with the quotient vector space $\mathfrak{g/h}$ , and the tangent space to any other $gH\in G/H$ is given by $\mathfrak{g}/g\mathfrak{h}g^{-1}$. The canonical class of the variety $G/H$ is  
\[
K_{G/H}= divisor(\overset{rs+s(s-r)}{\bigwedge} T^{\vee}_{G/H}). 
\]
We write the weight of the canonical differential $K_{G/H}$ as a product of those weights of $G$ which are not weights of $H$. 
\begin{thm}
{\it The weight of the canonical differential $K_{G/H}$ is $1$.}
\end{thm}

\begin{proof}
In total, there are $rs+s(s-r)$ weights that are weights of $G$ but not of $H$. We show that their product is $1$. 

Each asterisk block in the matrix below represents a collection of  weight spaces of $G$ that are not the weight spaces for $H$,
\[
\left(\begin{array}{c|c}
\ast_{r\times r} &  \ast_{r\times s-r} \\ \hline
\ast_{s-r\times r} &0_{s\times s}   \\ 
     
	\end{array}\right).
\]
The product of weights coming from the top left blok is equal to $1$. Also for every weight corresponding to a weight space in the top right block there is the reciprocal of that coming from the left bottom block. Hence  the product of all the  weights is $1$.

This shows that the weight of the canonical differential is $1$  under the action of $\IT_{H}$ on $\mathfrak{g/h}$.
\end{proof}
\section{Two special cases of $\GSV(r,s)$}
{\bf{Stiefel Varietiy:}} Let $X$  be a $r\times s$ matrix whose entries are real and $Y$ be the transpose of  $X$   given by 
\[
 X =
 \begin{pmatrix}
  x_{11} & x_{12} & \cdots & x_{1s} \\
  \vdots  & \vdots  & \ddots  & \vdots \\ 
  x_{r1} & x_{r2} & \cdots & x_{rs} \\
   \end{pmatrix}
   \qquad
\text{and }
\qquad
 Y =
 \begin{pmatrix}
  x_{11} & x_{21} & \cdots  & x_{r1} \\
   \vdots  & \vdots  & \ddots& \vdots    \\
  x_{1s} & x_{2s} & \cdots & x_{rs}\\
 \end{pmatrix}.
 \]
The  variety defined by $XY=XX^{tr}=I_{r\times r}$ is a Stiefel manifold.\\

{\bf{n-Sphere: }}If $X$ is a $1\times n+1$ row vector then the variety defined by $XX^{tr}=1$ is an $n$-sphere $S^{n}$. 
\section{Acknowledgements}
 This work was carried out when author was a mathematics research fellow at the International Center for Theoretical Physics(ICTP), Trieste, Italy. The author is very thankful to the Mathematics Section of the ICTP for awarding this fellowship. 


\end{document}